\documentclass[11pt,reqno]{amsart}
 \usepackage{amssymb,amsfonts,amscd,amsbsy}
\usepackage[mathscr]{eucal}
\usepackage{url}

\usepackage{amsmath}

\newtheorem{thm}{Theorem}[section]

\theoremstyle{definition}

\numberwithin{equation}{section}

\begin{document}

\title[$q$-series and tails of colored Jones polynomials]{$q$-series and tails of colored Jones polynomials}

\author{Paul Beirne and Robert Osburn}

\address{72 Oaktree Road, Merville, Stillorgan, County Dublin}

\address{School of Mathematics and Statistics, University College Dublin, Belfield, Dublin 4, Ireland}

\email{paul.beirne@ucdconnect.ie}

\email{robert.osburn@ucd.ie}


\date{\today}

\begin{abstract}
We extend the table of Garoufalidis, L{\^e} and Zagier concerning conjectural Rogers-Ramanujan type identities for tails of colored Jones polynomials to all alternating knots up to 10 crossings. We then prove these new identities using $q$-series techniques.
\end{abstract}

\maketitle

\section{Introduction}

The colored Jones polynomial $J_{N}(K; q)$ for a knot $K$ is an important quantum invariant of knots. Here, we use the normalization $J_{N}(K; q)=1$ for the unknot $K$, $J_{1}(K; q)=1$ for all knots $K$ and $J_{2}(K; q)$ is the Jones polynomial of $K$. The {\it tail} of $J_{N}(K; q)$ is a power series whose first $N$ coefficients agree (up to a common sign) with the first $N$ coefficients for $J_{N}(K; q)$ for all $N \geq 1$. If $K$ is an alternating knot, then the tail exists and equals an explicit $q$-multisum $\Phi_{K}(q)$ (see \cite{a}, \cite{dl}, \cite{gl}). 

Recently, Garoufalidis and L{\^e} (with Zagier) presented a table (see Table 6 in \cite{gl}) of 43 conjectural Rogers-Ramanujan type identities between the tails $\Phi_{K}(q)$ and products of theta functions and/or false theta functions. This table consisted of the following knots $K$: all alternating knots up to $8_4$, the twist knots $K_{p}$, $p>0$ or $p<0$, the torus knots $T(2, p)$, $p>0$, each of their mirror knots $-K$ and $-8_5$. For example, if we define for a positive integer $b$

\begin{equation*}
h_{b}= h_{b}(q) = \sum_{n \in \mathbb{Z}} \epsilon_{b}(n) q^{\frac{bn(n+1)}{2} - n}
\end{equation*}

\noindent where

\begin{equation*}
 \epsilon_{b}(n) = \left\{
  \begin{array}{ll}
    (-1)^n & \text{if $b$ is odd,}\\
    1 & \text{if $b$ is even and $n \geq 0$,} \\
    -1 & \text{if $b$ is even and $n < 0$} 
  \end{array} \right. 
\end{equation*}

\noindent and 

\begin{equation*}
(a)_n = (a;q)_n = \prod_{k=1}^{n} (1-aq^{k-1}),
\end{equation*}

\noindent valid for $n \in \mathbb{N} \cup \{\infty\}$, then

\begin{equation} \label{72}
\begin{aligned}
\Phi_{7_2}(q) & = (q)^{7}_{\infty} \sum_{a,b,c,d,e,f,g \geq 0}\frac{q^{3a^{2} + 2a + b^2 + bg + ac + ad + ae + af + ag + cd + de + ef + fg + c + d + e+ f + g}}{(q)_{a}(q)_{b}(q)_{c}(q)_{d}(q)_{e}(q)_{f}(q)_{g}(q)_{b+g}(q)_{a+c}(q)_{a+d}(q)_{a+e}(q)_{a+f}(q)_{a+g}} \\
& \stackrel{?}{=} h_{6}.
\end{aligned}
\end{equation}

\noindent Note that $h_1=0$, $h_2=1$ and $h_3 = (q)_{\infty}$. In general, $h_{b}$ is a theta function if $b$ is odd and a false theta function if $b$ is even. Using $q$-series techniques, Keilthy and the second author \cite{ko} proved not only (\ref{72}), but all of the remaining conjectural identities in \cite{gl}. 

The purpose of this paper is to extend the table of Garoufalidis, L{\^e} and Zagier to include all alternating knots up to 10 crossings. This is done in Tables \ref{89} and \ref{10} below. One immediately observes that their table is not ``complete" in the sense that there exist knots $K$ such that $\Phi_{K}(q) \neq \Phi_{K'}(q)$ for any knot $K'$ in Table 6 of \cite{gl}. For example, $\Phi_{8_7}(q) = h_3 h_5$. Our main result is the following.

\begin{thm} \label{main}
The identities in Tables \ref{89} and \ref{10} are true.
\end{thm}

\begin{table}[h]
\begin{tabular}{|c|c|c||c|c|c||c|c|c|}
\hline
$K$ & $\Phi_{K}(q)$ & $\Phi_{-K}(q)$ & $K$ & $\Phi_{K}(q)$ & $\Phi_{-K}(q)$ & $K$ & $\Phi_{K}(q)$ & $\Phi_{-K}(q)$\\ \hline
 $8_6$ &  $h_3 h_4$  &  $h_5$ & $9_6$ &  $h_3 h_6$  &  $h_4$ & $9_{24}$ & $?$ & $?$ \\
 $8_7$ &  $h_3 h_5$  &  $h_3^2$ & $9_7$ &  $h_3^{} h_4$  &  $h_6$ & $9_{25}$ & $h_3^3$ & $?$ \\
 $8_8$ &  $h_3 h_5$  &  $h_3^2$ & $9_8$ &  $h_3 h_6$  &  $h_3^2$ & $9_{26}$ & $h_3^2 h_4$ & $h_3^3$\\
 $8_9$ &  $h_3 h_4$  &  $h_3 h_4$ & $9_9$ &  $h_4 h_5$  &  $h_4$ & $9_{27}$ & $h_3^3$ & $h_3^2 h_4$\\
 $8_{10}$ &  ?  &  $h_3^2$ & $9_{10}$ &  $h_4^2$  &  $h_5$ & $9_{28}$ & $?$ & $?$ \\
 $8_{11}$ &  $h_3 h_4$  &  $h_3 h_4$ &$9_{11}$ &  $h_4^{} h_5$ & $h_3^2$ & $9_{29}$ & $?$ & $?$ \\
 $8_{12}$ &  $h_3^{} h_4$  &  $h_3^{} h_4$ & $9_{12}$ &  $h_3 h_4$  &  $h_3 h_5$ & $9_{30}$ & $h_3^3$ & $?$ \\
 $8_{13}$ &  $h_3^2 h_4$  &  $h_3^2$ & $9_{13}$ &  $h_4^2$  &  $h_3^{} h_4$ & $9_{31}$ & $h_3^4$ & $h_3^3$\\
 $8_{14}$ &  $h_3 h_4$  &  $h_3^3$ & $9_{14}$ &  $h_3^2 h_5$  &  $h_3^2$ & $9_{32}$ & $?$ & ?\\
 $8_{15}$ &  $h_3^3$  &  $?$ & $9_{15}$ &  $h_3 h_4$  &  $h_3 h_5$ & $9_{33}$ & ? & ?\\
 $8_{16}$ &  $?$  &  $?$ & $9_{16}$ &  $h_4$  &  $?$ & $9_{34}$ & ? & ?\\
 $8_{17}$ &  $?$  &  $?$ & $9_{17}$ &  $h_3^2$  &  $h_3^2 h_5$ &  $9_{35}$ & ? & $h_3$\\
 $8_{18}$ &  $?$  &  $?$ & $9_{18}$ &  $h_3 h_4$  &  $h_4^2$ & $9_{36}$ & ? & $h_3^2$\\
 $9_1$ &  $h_9$  &  $1$ & $9_{19}$ &  $h_3 h_5$  &  $h_3^3$ & $9_{37}$ & $h_3^3$ & ?\\
$9_2$ &  $h_8$  &  $h_3$ & $9_{20}$ &  $h_3^2$ &  $h_3 h_4^2$ & $9_{38}$ & ? & ?\\
$9_3$ &  $h_7$  &  $h_4$ & $9_{21}$ &  $h_3 h_4$  &  $h_3^2 h_4$ & $9_{39}$ & ? & ?\\
$9_4$ &  $h_6$  &  $h_5$ & $9_{22}$ & $?$ & $h_3^2$ & $9_{40}$ & ? & ?\\
$9_5$ &  $h_3$  &  $h_4 h_6$ & $9_{23}$ & $h_4^2$ & $h_3^3$ & $9_{41}$ & ? & ?\\
\hline
\end{tabular}
\caption{\label{89}}
\end{table}

\begin{table}[h]
\begin{tabular}{|c|c|c||c|c|c||c|c|c|}
\hline
$K$ & $\Phi_{K}(q)$ & $\Phi_{-K}(q)$ & $K$ & $\Phi_{K}(q)$ & $\Phi_{-K}(q)$ & $K$ & $\Phi_{K}(q)$ & $\Phi_{-K}(q)$\\ \hline
$10_1$ &  $h_9$  &  $h_3$ & $10_{27}$ & $h_3 h_5$ & $h_{3}^2 h_4$ & $10_{53}$ & ? & $h_3^3$  \\
$10_2$ &  $?$  &  $h_3$ & $10_{28}$ & $h_3 h_4 h_5$ & $h_3^{2}$ & $10_{54}$ & ? & $h_3^2$  \\
$10_3$ &  $h_7$  &  $h_5$ & $10_{29}$ & $h_3 h_4^{2}$ & $h_3 h_4$ & $10_{55}$ & ? & $h_3^3$  \\
$10_4$ &  $?$  &  $h_3$ & $10_{30}$ &  $h_3 h_4^{2}$ & $h_3^3$ & $10_{56}$ & ? & $h_3 h_4$ \\
$10_5$ &  $h_3 h_7$  &  $h_3^2$ & $10_{31}$ & $h_3 h_5$ & $h_3^2 h_4$ & $10_{57}$ & ? & $h_3^2 h_4$ \\
$10_6$ &  $h_3 h_6$  &  $h_5$  & $10_{32}$ & $?$ & $h_3^3$ & $10_{58}$ & ? & $h_3^3$ \\
$10_7$ &  $h_3^{} h_6$  &  $h_3^{} h_4$ & $10_{33}$ & $?$ & $h_3^{2} h_4$ & $10_{59}$ & ? & $h_3^3$ \\
$10_8$ &  $h_3^{}$  &  $h_5 h_6$ & $10_{34}$ & $h_3 h_7$ & $h_3^2$ & $10_{60}$ & ? & $h_3^3$ \\
$10_9$ &  $h_3 h_6$  &  $h_3^{} h_4$ & $10_{35}$ & $h_3 h_6$ & $h_3 h_4$ & $10_{61}$ & ? & $h_3$ \\
$10_{10}$ &  $h_3^2 h_6$  &  $h_3^2$ & $10_{36}$ & $h_3 h_6$ & $h_3^3$ & $10_{62}$ & ? & $h_3^2$ \\
$10_{11}$ &  $h_4 h_5$  &  $h_5$ & $10_{37}$ & $h_3 h_5$ & $h_3 h_5$ & $10_{63}$ & ? & $h_3 h_4$ \\ 
$10_{12}$ & $h_3 h_5$ &  $h_3 h_5$ & $10_{38}$ & $?$ & $h_3^3$ & $10_{64}$ & ? & $h_3 h_4$ \\
$10_{13}$ &  $h_4 h_5$  &  $h_3 h_4$ & $10_{39}$ & $h_3 h_4$ & $h_3^2 h_5$ & $10_{65}$ & ? & $h_3^2 h_4$ \\ 
$10_{14}$ & $h_3^2 h_5$ & $h_3 h_4$ & $10_{40}$ & $?$ & $h_3^2 h_4$ & $10_{66}$ & ? & $?$ \\
$10_{15}$ & $h_5^2$ & $h_3^2$ & $10_{41}$ & $h_3 h_4^2$ & $h_3^3$ & $10_{67}$ & ? & $h_3^3$ \\
$10_{16}$ & $h_4 h_5$ & $h_3 h_4$ & $10_{42}$ & $h_3^2 h_4$ & $?$ & $10_{68}$ & ? & $h_3^2$ \\
$10_{17}$ & $?$ & $h_3 h_5$ & $10_{43}$ & $h_3^2 h_4$ & $h_3^2 h_4$ & $10_{69}$ & ? & $?$ \\
$10_{18}$ & $h_3^2 h_5$ & $h_3 h_4$ & $10_{44}$ & $h_3^3 h_4$ & $h_3^4$ & $10_{70}$ & $?$ & $h_3 h_4$ \\
$10_{19}$ & $h_3 h_4 h_5$ & $h_3^2$ & $10_{45}$ & $h_3^4$ & $h_3^4$ & $10_{71}$ & ? & $h_3^2 h_4$ \\
$10_{20}$ & $h_7$ & $h_3 h_4$ & $10_{46}$ & ? & $h_3$& $10_{72}$ & $h_3 h_4$ & ? \\
$10_{21}$ & $h_3 h_6$ & $h_3 h_4$ & $10_{47}$ & ? & $h_3^2$ & $10_{73}$ & ? & $h_3^2 h_4$ \\
$10_{22}$ & $h_3 h_4$ & $h_4 h_5$ & $10_{48}$ & ? & $h_3 h_5$ & $10_{74}$ & ? & $h_3 h_4$ \\
$10_{23}$ & $h_3 h_5$ & $h_3^2 h_4$ & $10_{49}$ & ? &$h_3^2 h_5$  & $10_{75}$ & $?$ & ? \\
$10_{24}$ & $h_4 h_5$ & $h_3 h_4$ & $10_{50}$ & ? & $h_3 h_4$ & $10_{76}$ & ? & $h_5$ \\
$10_{25}$ & $h_3 h_4^2$ & $h_3 h_4$ & $10_{51}$ & ? & $h_3^2 h_4$ & $10_{77}$ & ? & $h_3 h_5$ \\
$10_{26}$ &$h_3 h_4^2$  &$h_3 h_4$ & $10_{52}$ &? & $h_3^3$  & $10_{78}$ & ? & ? \\
\hline
\end{tabular}
\caption{\label{10}}
\end{table}

Unfortunately, we were unable to find similar identities not only in each case labelled ``$?$" in Tables \ref{89} and \ref{10}, but for any alternating knot (or its mirror) from $10_{79}$ to $10_{123}$. This is also the situation for $8_5$ where although one has (after $q$-theoretic simplification or the methods in \cite{hbubble})

\begin{equation} \label{85}
\Phi_{8_5}(q) = (q)_{\infty}^2 \sum_{a,b \geq 0} \frac{q^{a^2 + a + b^2 + b} (q)_{a+b}}{(q)_a^2 (q)_b^2}, 
\end{equation}

\noindent the modular (or false theta, mock/mixed mock, quantum modular) properties of the double sum in (\ref{85}) are not clear. The difficulty in finding nice identities for these tails is due to the structure of their reduced Tait graphs (see \cite{gv}). Another approach to Theorem \ref{main} is to utilize the skein-theoretic techniques in \cite{ad}, \cite{eh} and \cite{h}. It would be of considerable interest to investigate the connection between skein theory and $q$-series to gain a better understanding of these unknown cases and of a general framework.

It would also be desirable to study $q$-series identities in other settings which arise from knot theory. For example, the $q$-multisum $\Phi_{K}(q)$ occurs as the ``$0$-limit" of $J_{N}(K; q)$ (see Theorem 2 in \cite{gl}). Garoufalidis and L{\^e} have also obtained an explicit formula (see Theorem 3 in \cite{gl}) for the ``1-limit" of $J_{N}(K; q)$. Finally, do tails exist (in some appropriate sense) for generalizations of $J_{N}(K; q)$ (see \cite{gnsss}, \cite{no}--\cite{veen})?

The paper is organized as follows. In Section 2, we recall the necessary background from \cite{ko}. In Section 3, we prove Theorem \ref{main}. 

\section{Preliminaries}

We first recall six $q$-series identities (see (2.1)--(2.3), Lemma 2.1, (4.3) and the proof of (4.1) in \cite{ko}). Namely, 

\begin{equation} \label{e1}
\sum_{n=0}^{\infty} \frac{t^n}{(q)_{n}} = \frac{1}{(t)_{\infty}},
\end{equation}

\begin{equation} \label{e2}
\sum_{n=0}^{\infty} \frac{(-1)^n t^n q^{n(n-1)/2}}{(q)_{n}} = (t)_{\infty},
\end{equation}

\begin{equation} \label{andy} 
\sum_{n=0}^{\infty} \frac{q^{n^2 + An}}{(q)_{n} (q)_{n+A}} = \frac{1}{(q)_{\infty}}
\end{equation}

\noindent for any integer $A$,

\begin{equation} \label{double}
\sum_{m,n \geq 0} (-1)^{n} \frac{q^{m^2 + m + mn + \frac{n(n+1)}{2}}}{(q)_m (q)_n} = h_4,
\end{equation}

\begin{equation} \label{triple}
\sum_{l, m, n \geq 0} (-1)^{l+n} \frac{q^{\frac{3l(l+1)}{2} + m^2 + m + \frac{n(n+1)}{2} + 2lm + ln + mn}}{(q)_l (q)_m (q)_n} = h_5
\end{equation}

\noindent and

\begin{equation} \label{key}
\sum_{a \geq 0} (-1)^{na}  \frac{q^{\frac{na(a+1)}{2} -a + a  \sum\limits_{k=1}^{n-1} c_{k}}}{(q)_{a}\prod\limits_{k=1}^{n-1}(q)_{a+c_{k}}} = \frac{1}{(q)_{\infty}} \sum_{i_{1}, \dotsc, i_{n-2} \geq 0} (-1)^{\sum\limits_{k=1}^{n-2}\sum\limits_{j=1}^{k}i_{j}} \frac{q^{\frac{1}{2}\sum\limits_{k=1}^{n-2}\bigl(\sum\limits_{j=1}^{k}i_{j}\bigr)\bigl(1+\sum\limits_{j=1}^{k}i_{j} \bigr)+ \sum\limits_{k=2}^{n-1}\sum\limits_{j=1}^{k-1}c_{k}i_{j}}}{\prod\limits_{k=1}^{n-2}(q)_{i_{k}}\prod\limits_{k=1}^{n-2} (q)_{c_{k} +\sum\limits_{j=1}^{k}i_{j} }}
\end{equation}

\noindent for any $n>2$ and integers $c_k$.

Let $K$ be an alternating knot with $c$ crossings and $\mathcal{T}_{K}$ its associated Tait graph. The reduced Tait graph $\mathcal{T}_{K}^{\prime}$ is obtained from $\mathcal{T}_{K}$ by replacing every set of two edges that connect the same two vertices by a single edge. The tail $\Phi_{K}(q)$ is given by

\begin{equation}
\Phi_{K}(q) = (q)_{\infty}^{c} S_{K}(q)
\end{equation}

\noindent where $S_{K}(q)$ is an explicitly constructed $q$-multisum (see pages 261--264 in \cite{ko}). Now, by Theorem 2 in \cite{ad}, if $\mathcal{T}_{K}^{\prime}$ is the same as $\mathcal{T}_{L}^{\prime}$ for two alternating knots $K$ and $L$, then $\Phi_{K}(q) = \Phi_{L}(q)$. Thus, by comparing the reduced Tait graphs for those knots in Table 1 of \cite{ko} and Tables \ref{89} and \ref{10} above, it suffices to verify the conjectural identities in the following cases: $8_7$, $8_{13}$, $-9_5$, $9_{14}$, $-9_{17}$, $-9_{20}$, $-9_{27}$, $9_{31}$, $10_{5}$, $-10_{8}$, $10_{10}$, $10_{15}$, $10_{19}$, $10_{26}$, $10_{28}$, $10_{44}$. Note that Corollary 2 in \cite{gl} is false as stated since $\mathcal{T}_{8_6}^{\prime} \cong \mathcal{T}_{9_{24}}^{\prime}$, but $\Phi_{8_6}(q) \neq \Phi_{9_{24}}(q)$. 

The strategy for proving Theorem \ref{main} is now as follows. For each of the 16 cases, we first compute $S_{K}(q)$ using the methods from \cite{ko}. We then employ (\ref{e1})--(\ref{key}) to reduce this $q$-multisum to (\ref{72}) or one of the following key identities proven in \cite{ko}:

\begin{equation} \label{51}
S_{5_1}(q) :=\sum_{a,b,c,d,e \geq 0} (-1)^{a}\frac{q^{\frac{a(5a+3)}{2} + ab + ac + ad + ae + bc + cd + de + b + c + d + e}}{(q)_{a}(q)_{b}(q)_{c}(q)_{d}(q)_{e}(q)_{a+b}(q)_{a+c}(q)_{a+d}(q)_{a+e}} = \frac{1}{(q)^5_{\infty}} h_5,
\end{equation}

\begin{equation} \label{62}
S_{6_2}(q) := \sum_{a,b,c,d,e,f \geq 0}(-1)^{e} \frac{q^{2f^{2} + f + \frac{e(3e+1)}{2} + ab + af + bc + bf + cd + ce + cf + de + a + b + c + d}}{(q)_{a}(q)_{b}(q)_{c}(q)_{d}(q)_{e}(q)_{f}(q)_{a+f}(q)_{b+f}(q)_{c+e}(q)_{c+f}(q)_{d+e}} = \frac{1}{(q)^5_{\infty}} h_4,
\end{equation}

\begin{equation} \label{71}
\begin{aligned}
S_{7_1}(q) &:=\sum_{a,b,c,d,e,f,g \geq 0} (-1)^{a}\frac{q^{\frac{a(7a+5)}{2} + ab + ac + ad + ae + af + ag + bc + cd + de + ef + fg + b + c + d + e+ f + g}}{(q)_{a}(q)_{b}(q)_{c}(q)_{d}(q)_{e}(q)_{f}(q)_{g}(q)_{a+b}(q)_{a+c}(q)_{a+d}(q)_{a+e}(q)_{a+f}(q)_{a+g}} \\
& = \frac{1}{(q)^7_{\infty}} h_7,
\end{aligned}
\end{equation}

\begin{equation} \label{74} 
\begin{aligned}
S_{7_4}(q) & := \sum_{a,b,c,d,e,f,g \geq 0} \frac{q^{2f^{2} + f +2g^{2} + g  + ab + ag + bc + bg + cd + cf + cg + de +df + ef + a + b + c + d + e}}{(q)_{a}(q)_{b}(q)_{c}(q)_{d}(q)_{e}(q)_{f}(q)_{g}(q)_{a+g}(q)_{b+g}(q)_{c+f}(q)_{c+g}(q)_{d+f}(q)_{e+f}} \\
& = \frac{1}{(q)^7_{\infty}} h_4^2,
\end{aligned}
\end{equation}

\begin{equation} \label{77}
\begin{aligned}
& S_{7_7}(q) := \sum_{a,b,c,d,e,f,g \geq 0} (-1)^{e+f+g}\frac{q^{\frac{3e^2}{2} + \frac{e}{2} +\frac{3f^2}{2} + \frac{f}{2} +\frac{3g^2}{2} + \frac{g}{2} + ab + ad + ae + af + bf + cd + cg + de + dg + a + b + c}}{(q)_{a}(q)_{b}(q)_{c}(q)_{d}(q)_{e}(q)_{f}(q)_{g}(q)_{a+e}(q)_{d+e}(q)_{a+f}(q)_{b+f}(q)_{c+g}} \\
& \times \frac{q^{d}}{(q)_{d+g}} \\
& =  \frac{1}{(q)_{\infty}^4},
\end{aligned}
\end{equation}

\begin{equation} \label{82}
\begin{aligned}
& S_{8_2}(q) :=\sum_{a,b,c,d,e,f,g,h \geq 0}(-1)^{b}\frac{q^{3a^{2} + 2a + \frac{b(3b+1)}{2} + ad + ae + af + ag + ah + bc + bd + cd + de + ef + fg + gh + c + d + e+ f}}{(q)_{a}(q)_{b}(q)_{c}(q)_{d}(q)_{e}(q)_{f}(q)_{g}(q)_{h}(q)_{b+c}(q)_{b+d}(q)_{a+d}(q)_{a+e}(q)_{a+f}} \\
& \times \frac{q^{g+h}}{(q)_{a+g} (q)_{a+h}} \\
& = \frac{1}{(q)^7_{\infty}} h_6
\end{aligned}
\end{equation}

\noindent and

\begin{equation} \label{-84}
\begin{aligned}
& S_{-8_4}(q) := \sum_{a,b,c,d,e,f,g,h \geq 0}(-1)^{g}\frac{q^{\frac{g(5g+3)}{2} + h(2h+1) + ab + ah + bc + bh + cd + cg + ch + de + dg + ef + eg + fg + a + b + c + d}}{(q)_{a}(q)_{b}(q)_{c}(q)_{d}(q)_{e}(q)_{f}(q)_{g}(q)_{h}(q)_{a+h}(q)_{b+h}(q)_{c+g}(q)_{c+h}(q)_{d+g}} \\
& \times \frac{q^{e+f}}{(q)_{e+g} (q)_{f+g}} \\
& = \frac{1}{(q)^{8}_{\infty}} h_4 h_5. 
\end{aligned}
\end{equation}

\section{Proof of Theorem \ref{main}}

\begin{proof}[Proof of Theorem \ref{main}]
We give full details for $8_7$, $-9_5$ and $-10_8$. As the remaining cases are handled similarly, we sketch their proofs. For $\Phi_{8_7}(q)$, it suffices to prove 

\begin{equation} \label{87}
\begin{aligned}
& S_{8_7}(q) := \sum_{a,b,c,d,e,g,h,i \geq 0} (-1)^{h+i} \frac{q^{\frac{i(5i+3)}{2} + \frac{h(3h+1)}{2} + g^2 + ab + ag + ah + bc + bh + bi + cd + ci + de + di + ei + a + b + c}}{(q)_a (q)_b (q)_c (q)_d (q)_e (q)_g (q)_h (q)_i (q)_{a+g} (q)_{a+h} (q)_{b+h} (q)_{b+i} (q)_{c+i}} \\
& \times \frac{q^{d+e}}{(q)_{d+i} (q)_{e+i}} \\
& = \frac{1}{(q)_{\infty}^{7}} h_5.
\end{aligned}
\end{equation}

\noindent We now have

\begin{equation*}
\begin{aligned}
& S_{8_7}(q) = \frac{1}{(q)_{\infty}} \sum_{a,b,c,d,e,h,i \geq 0}  (-1)^{h+i} \frac{q^{\frac{i(5i+3)}{2} + \frac{h(3h+1)}{2} + ab + ah + bc + bh + bi + cd + ci + de + di + ei + a + b + c}}{(q)_a (q)_b (q)_c (q)_d (q)_e (q)_h (q)_i (q)_{a+h} (q)_{b+h} (q)_{b+i} (q)_{c+i} (q)_{d+i}} \\
& \times \frac{q^{d+e}}{(q)_{e+i}} \\
& (\text{evaluate the $g$-sum with (\ref{andy})}) \\
& = \frac{1}{(q)_{\infty}^2} \sum_{a,b,c,d,e,h,i \geq 0}  (-1)^{h+i} \frac{q^{\frac{i(5i+3)}{2} + \frac{h(h+1)}{2} + ab + ah + bc + bi + cd + ci + de + di + ei + a + b + c + d + e}}{(q)_a (q)_b (q)_c (q)_d (q)_e (q)_h (q)_i (q)_{b+h} (q)_{b+i} (q)_{c+i} (q)_{d+i} (q)_{e+i}} \\
& (\text{apply (\ref{key}) to the $h$-sum with $n=3$}) \\
& = \frac{1}{(q)_{\infty}^2} \sum_{b,c,d,e,i \geq 0} (-1)^i \frac{q^{\frac{i(5i+3)}{2} + bc + bi + cd + ci + de + di + ei + b + c + d + e}}{(q)_{b} (q)_c (q)_d (q)_e (q)_i (q)_{b+i} (q)_{c+i} (q)_{d+i} (q)_{e+i}} \\
& (\text{evaluate the $a$-sum with (\ref{e1}), simplify, then use (\ref{e2}) for the $h$-sum}).
\end{aligned}
\end{equation*}

\noindent Thus, (\ref{87}) then follows from (\ref{51}) after letting $i \to a$.

For $\Phi_{8_{13}}(q)$, it suffices to prove

\begin{equation} \label{813}
\begin{aligned}
S_{8_{13}}(q) & := \sum_{a,c,d,e,f,g,h,i \geq 0} (-1)^{g+h} \frac{q^{\frac{g(3g+1)}{2} + \frac{(3h+1)}{2} + i(2i+1) + af + ag + ci + cd + de + di + ef + eh + ei}}{(q)_a (q)_c (q)_d (q)_e (q)_f (q)_g (q)_h (q)_i (q)_{a+g} (q)_{c+i} (q)_{d+i} (q)_{e+i} (q)_{e+h}} \\
& \times \frac{q^{fh + fg + a + c + d + e + f}}{(q)_{f+h} (q)_{f+g}} \\
& = \frac{1}{(q)_{\infty}^6} h_4.
\end{aligned}
\end{equation}

\noindent Apply (\ref{key}) with $n=3$ to the $g$-sum, (\ref{e1}) to the $a$-sum, then simplify and (\ref{e2}) to the $g$-sum to obtain

\begin{equation*}
S_{8_{13}}(q) = \frac{1}{(q)_{\infty}} \sum_{c,d,e,f,h,i \geq 0} (-1)^h \frac{q^{\frac{h(3h+1)}{2} + i(2i+1) +  ci + cd + de + di + ef + eh + ei + fh + c + d + e + f}}{(q)_{c} (q)_d (q)_e (q)_f (q)_h (q)_i (q)_{c+i} (q)_{d+i} (q)_{e+i} (q)_{e+h} (q)_{f+h}}.
\end{equation*}

\noindent Thus, (\ref{813}) then follows from (\ref{62}) upon $(c,d,e,f,h,i) \to (a,b,c,d,e,f)$. 

For $\Phi_{-9_5}(q)$, it suffices to prove

\begin{equation} \label{-95}
\begin{aligned}
& S_{-9_5}(q) := \sum_{a,b,c,d,e,f,g,h,j \geq 0} \frac{q^{h(2h+1) + j(3j+2) + ab + ag + ah + aj + bc + bh + ch + de + dj + ef + ej + fg + fj + gj + a + b + c}}{(q)_a (q)_b (q)_c (q)_d (q)_e (q)_f (q)_g (q)_h (q)_j (q)_{a+h} (q)_{a+j} (q)_{b+h} (q)_{c+h} (q)_{d+j}} \\
& \times \frac{q^{d+e+f+g}}{(q)_{e+j} (q)_{f+j} (q)_{g+j}} \\
& = \frac{1}{(q)_{\infty}^9} h_4 h_6.
\end{aligned}
\end{equation}

\noindent We now have

\begin{equation*}
\begin{aligned}
& S_{-9_5}(q) = \frac{1}{(q)_{\infty}} \sum_{a,b,c,d,e,f,g,j,s,t \geq 0} \frac{q^{s^2 + s + st + \frac{t(t+1)}{2} + bs + c(s+t) + j(3j+2) + ab + ag + aj + bc + de + dj + ef + ej + fg}}{(q)_a (q)_b (q)_c (q)_d (q)_e (q)_f (q)_g (q)_j (q)_s (q)_t (q)_{a+j} (q)_{d+j} (q)_{s+a}} \\
& \times \frac{q^{fj + gj + a + b + c + d + e + f + g}}{(q)_{s+t+b} (q)_{e+j} (q)_{f+j} (q)_{g+j}} \\
&( \text{apply (\ref{key}) to the $h$-sum with $n=4$}) \\
& = \frac{1}{(q)_{\infty}^2} \sum_{a,b,d,e,f,g,j,s,t \geq 0} \frac{q^{s^2 + s + st + \frac{t(t+1)}{2} + bs + j(3j+2) + ab + ag + aj + de + dj + ef + ej + fg + fj + gj + a + b + d + e}}{(q)_a (q)_b (q)_d (q)_e (q)_f (q)_g (q)_j (q)_s (q)_t (q)_{a+j} (q)_{d+j} (q)_{e+j} (q)_{f+j} (q)_{g+j}} \\
& \times \frac{q^{f+g}}{(q)_{s+a}} \\
& (\text{evaluate the $c$-sum with (\ref{e1}) and simplify}) \\
\end{aligned}
\end{equation*}

\begin{equation*}
\begin{aligned}
& = \frac{1}{(q)_{\infty}^3} h_4 \sum_{a,d,e,g,j \geq 0} \frac{q^{j(3j+2) + ag + aj + de + dj + ef + ej + fg + fj + gj + a + d + e + f + g}}{(q)_a (q)_d (q)_e (q)_f (q)_f (q)_g (q)_j (q)_{a+j} (q)_{d+j} (q)_{e+j} (q)_{f+j} (q)_{g+j}} \\
& (\text{evaluate the $b$-sum with (\ref{e1}), simplify, then apply (\ref{double}) to the $st$-sum}).
\end{aligned}
\end{equation*}

\noindent Now, (\ref{-95}) follows from first applying (\ref{andy}) the $b$-sum in (\ref{72}), then letting $(a,d,e,f,g,j) \to (c,g,f,e,d,a)$. 

For $\Phi_{9_{14}}(q)$, it suffices to prove

\begin{equation} \label{914}
\begin{aligned}
S_{9_{14}}(q) & := \sum_{a,b,c,d,e,g,h,i,j \geq 0} (-1)^{h+i+j} \frac{q^{\frac{h(3h+1)}{2} + \frac{i(3i+1)}{2} + \frac{j(5j+3)}{2} + ab + ag + ah + ai + bc + bi + bj + cd + cj + de + dj + ej}}{(q)_{a} (q)_b (q)_c (q)_d (q)_e (q)_g (q)_h (q)_i (q)_j (q)_{a+h} (q)_{a+i} (q)_{b+i} (q)_{b+j}} \\
& \times \frac{q^{gh + a + b + c + d + e +g}}{(q)_{c+j} (q)_{d+j} (q)_{e+j} (q)_{g+h}} \\
& = \frac{1}{(q)_{\infty}^7} h_5.
\end{aligned}
\end{equation}

\noindent First, apply (\ref{key}) with $n=3$ to the $h$-sum, (\ref{e1}) to the $g$-sum, simplify and (\ref{e2}) to the $h$-sum, then (\ref{key}) with $n=3$ to the $i$-sum, (\ref{e1}) to the $a$-sum, simplify and (\ref{e2}) to the $i$-sum to obtain

\begin{equation*}
S_{9_{14}}(q) = \frac{1}{(q)_{\infty}^2} \sum_{b,c,d,e,j} (-1)^{j} \frac{q^{\frac{j(5j+3)}{2} + bc + bj + cd + cj + de + dj + ej + b + c + d + e}}{(q)_b (q)_c (q)_d (q)_e (q)_j (q)_{b+j} (q)_{c+j} (q)_{d+j} (q)_{e+j}}.
\end{equation*}

\noindent Thus, (\ref{914}) follows from (\ref{51}) after $j \to a$.

For $\Phi_{-9_{17}}(q)$, it suffices to prove

\begin{equation} \label{-917}
\begin{aligned}
& S_{-9_{17}}(q) := \sum_{a,b,c,d,e,f,h,i,j \geq 0} (-1)^{h+i+j} \frac{q^{\frac{h(3h+1)}{2} + \frac{i(5i+3)}{2} + \frac{j(3j+1)}{2} + ab + aj + bc + bi + bj + cd + ci + de + di + ef}}{(q)_a (q)_b (q)_c (q)_d (q)_e (q)_f (q)_h (q)_i (q)_j (q)_{a+j} (q)_{b+i} (q)_{b+j} (q)_{c+i}} \\
& \times \frac{q^{eh + ei + fh + a + b + c + d + e + f}}{(q)_{d+i} (q)_{e+h} (q)_{e+i} (q)_{f+h}} \\
& = \frac{1}{(q)_{\infty}^7} h_5.
\end{aligned}
\end{equation}

\noindent First, apply (\ref{key}) with $n=3$ to the $h$-sum, (\ref{e1}) to the $f$-sum, simplify and (\ref{andy}) to the $h$-sum, then (\ref{key}) with $n=3$ to the $j$-sum, (\ref{e1}) to the $a$-sum, simplify and (\ref{andy}) to the $j$-sum to get

\begin{equation*}
S_{-9_{17}}(q) = \frac{1}{(q)_{\infty}^2} \sum_{b, c, d, e, i \geq 0} (-1)^i \frac{q^{\frac{i(5i+3)}{2} + bc + bi + cd + ci + de + di + ei + b + c + d + e}}{(q)_b (q)_c (q)_d (q)_e (q)_i (q)_{b+i} (q)_{c+i} (q)_{d+i} (q)_{e+i}}.
\end{equation*}

\noindent Thus, (\ref{-917}) follows from (\ref{51}) after $i \to a$.

For $\Phi_{-9_{20}}(q)$, it suffices to prove

\begin{equation} \label{-920}
\begin{aligned}
S_{-9_{20}}(q) & := \sum_{a,b,c,d,e,f,h,i,j \geq 0} (-1)^h \frac{q^{\frac{h(3h+1)}{2} + i(2i+1) + j(2j+1) + ab + ah + bc + bh + bi + cd + ci + de + di + dj + ef + ej}}{(q)_a (q)_b (q)_c (q)_d (q)_e (q)_f (q)_h (q)_i (q)_j (q)_{a+h} (q)_{b+h} (q)_{b+i} (q)_{c+i}} \\
& \times \frac{q^{fj + a + b + c + d + e + f}}{(q)_{d+i} (q)_{d+j} (q)_{e+j} (q)_{f+j}} \\
& = \frac{1}{(q)_{\infty}^8} h_4^2.
\end{aligned}
\end{equation}

\noindent Apply (\ref{key}) with $n=3$ to the $h$-sum, (\ref{e1}) to the $a$-sum and simplify, then (\ref{e2}) to the $h$-sum to obtain

\begin{equation*}
S_{-9_{20}}(q) = \frac{1}{(q)_{\infty}} \sum_{b,c,d,e,f,i,j \geq 0} \frac{q^{i(2i+1) + j(2j+1) + bc + bi + cd + ci + de + di + dj + ef + ej + fj + b + c + d + e + f}}{(q)_b (q)_c (q)_d (q)_e (q)_f (q)_i (q)_j (q)_{b+i} (q)_{c+i} (q)_{d+i} (q)_{d+j} (q)_{e+j} (q)_{f+j}}.
\end{equation*}

\noindent Now, (\ref{-920}) follows from (\ref{74}) after the substitution $(b,c,d,e,f,i,j) \to (a,b,c,d,e,g,f)$. 

For $\Phi_{-9_{27}}(q)$, it suffices to prove

\begin{equation} \label{-927}
\begin{aligned}
S_{-9_{27}}(q) & := \sum_{a,b,c,d,e,f,g,h,i \geq 0} (-1)^{f+h} \frac{q^{\frac{f(3f+1)}{2} + g(2g+1) + \frac{h(3h+1)}{2} + i^2 + ab + af + bc + bf + bg + cd + cg + de + dg + dh}}{(q)_a (q)_b (q)_c (q)_d (q)_e (q)_f (q)_g (q)_h (q)_i (q)_{a+f} (q)_{b+f} (q)_{b+g} (q)_{c+g}} \\
& \times \frac{q^{eh + ei + a + b + c + d + e}}{(q)_{d+g} (q)_{d+h} (q)_{e+h} (q)_{e+i}} \\
& = \frac{1}{(q)_{\infty}^7} h_4.
\end{aligned}
\end{equation}

\noindent Apply (\ref{andy}) to the $i$-sum, (\ref{key}) with $n=3$ to the $f$-sum, (\ref{e1}) to the $a$-sum, simplify and (\ref{e2}) to the $f$-sum to obtain

\begin{equation*}
S_{-9_{27}} = \frac{1}{(q)_{\infty}^2} \sum_{b,c,d,e,g,h \geq 0} (-1)^h \frac{q^{g(2g+1) + \frac{h(3h+1)}{2} + bc + bg + cd + cg + de + dg + dh + eh + b + c + d + e}}{(q)_b (q)_c (q)_d (q)_e (q)_g (q)_h (q)_{b+g} (q)_{c+g} (q)_{d+g} (q)_{d+h} (q)_{e+h}}.
\end{equation*}

\noindent Now, (\ref{-927}) follows from (\ref{62}) after letting $(b,c,d,e,g,h) \to (a,b,c,d,f,e)$. 

For $\Phi_{9_{31}}(q)$, it suffices to prove

\begin{equation} \label{931}
\begin{aligned}
& S_{9_{31}}(q) := \sum_{a,b,c,e,f,g,h,i,j \geq 0} (-1)^{g+h+i+j} \frac{q^{\frac{g(3g+1)}{2} + \frac{h(3h+1)}{2} + \frac{i(3i+1)}{2} + \frac{j(3j+1)}{2} + ab + af + ag + aj + bc + bg + bh}}{(q)_{a} (q)_b (q)_c (q)_e (q)_f (q)_g (q)_h (q)_i (q)_j (q)_{a+g} (q)_{a+j} (q)_{b+g}} \\
& \times \frac{q^{ch + ef + ei + fi + fj + a + b + c + e + f}}{(q)_{b+h} (q)_{c+h} (q)_{e+i} (q)_{f+i} (q)_{f+j}} \\
& = \frac{1}{(q)_{\infty}^5}.
\end{aligned}
\end{equation}

\noindent Apply (\ref{key}) with $n=3$ to the $h$-sum, (\ref{e1}) to the $c$-sum, simplify and (\ref{e2}) to the $h$-sum to obtain

\begin{equation*}
\begin{aligned}
& S_{9_{31}}(q)  = \frac{1}{(q)_{\infty}} \sum_{a,b,e,f,g,i,j \geq 0} (-1)^{g+i+j} \frac{q^{\frac{g(3g+1)}{2} + \frac{i(3i+1)}{2} + \frac{j(3j+1)}{2} + ab + af + ag + aj + bg + ef + ei + fi + fj + a}}{(q)_a (q)_b (q)_e (q)_f (q)_g (q)_i (q)_j (q)_{a+g} (q)_{a+j} (q)_{b+g} (q)_{e+i} (q)_{f+i}} \\
& \times \frac{q^{b+e+f}}{(q)_{f+j}}.
\end{aligned}
\end{equation*}

\noindent Now, (\ref{931}) follows from (\ref{77}) after letting $(a,b,e,f,g,i,j) \to (a,b,c,d,f,g,e)$. 

For $\Phi_{10_5}(q)$, it suffices to prove 

\begin{equation} \label{105}
\begin{aligned}
& S_{10_5}(q) := \sum_{a,b,c,d,e,f,g,i,j,k \geq 0} (-1)^{j+k} \frac{q^{\frac{j(3j+1)}{2} + \frac{k(7k+5)}{2} + i^2 + ab + ai + aj + bc + bj + bk + cd + ck + de + dk + ef + ek + fg}}{(q)_a (q)_b (q)_c (q)_d (q)_e (q)_f (q)_g (q)_i (q)_j (q)_k (q)_{a+i} (q)_{a+j} (q)_{b+j}} \\
& \times \frac{q^{fk + gk + a + b + c + d + e + f + g}}{(q)_{b+k} (q)_{c+k} (q)_{d+k} (q)_{e+k} (q)_{f+k} (q)_{g+k}} \\
& = \frac{1}{(q)_{\infty}^9} h_7.
\end{aligned}
\end{equation}

\noindent Apply (\ref{andy}) to the $i$-sum, (\ref{key}) with $n=3$ to the $j$-sum, (\ref{e1}) to the $a$-sum and simplify, then (\ref{e2}) to the $j$-sum to obtain

\begin{equation*}
\begin{aligned}
& S_{10_5}(q) = \frac{1}{(q)_{\infty}^2} \sum_{b,c,d,e,f,g,k \geq 0} (-1)^k \frac{q^{\frac{k(7k+5)}{2} + bc + bk + cd + ck + de + dk + ef + ek + fg + fk + gk + b + c + d + e + f}}{(q)_b (q)_c (q)_d (q)_e (q)_f (q)_g (q)_k (q)_{b+k} (q)_{c+k} (q)_{d+k} (q)_{e+k} (q)_{f+k}} \\
& \times \frac{q^{g}}{(q)_{g+k}}.
\end{aligned}
\end{equation*}

\noindent Now, (\ref{105}) follows from (\ref{71}) after letting $k \to a$. 

For $\Phi_{-10_8}(q)$, it suffices to prove

\begin{equation} \label{-108}
\begin{aligned}
& S_{-10_8}(q) := \sum_{a,b,c,d,e,f,g,h,i,k \geq 0} (-1)^{i} \frac{q^{\frac{i(5i+3)}{2} + k(3k+2) + ab + ae + ai + ak + bc + bi + cd + ci + di + ef + ek + fg + fk + gh}}{(q)_a (q)_b (q)_c (q)_d (q)_e (q)_f (q)_g (q)_h (q)_i (q)_k (q)_{a+i} (q)_{a+k} (q)_{b+i}} \\
& \times \frac{q^{gk + hk + a+b+c+d+e+f+g+h}}{(q)_{c+i} (q)_{d+i} (q)_{e+k} (q)_{f+k} (q)_{g+k} (q)_{h+k}} \\
& = \frac{1}{(q)_{\infty}^{10}} h_5 h_6.
\end{aligned}
\end{equation}

\noindent We now have

\begin{equation*}
\begin{aligned}
& S_{-10_8}(q) = \frac{1}{(q)_{\infty}} \sum_{a,b,c,d,e,f,g,h,i,k, j, l \geq 0} (-1)^{i+l} \frac{q^{\frac{3i(i+1)}{2} + j^2 + j + \frac{l(l+1)}{2} + 2ij + il + jl + k(3k+2) + ab + ae  + ak + bc + bi}}{(q)_a (q)_b (q)_c (q)_d (q)_e (q)_f (q)_g (q)_h (q)_i (q)_j (q)_k (q)_l} \\
& \times \frac{q^{cd + c(i+j) + d(i+j+l) + ef + ek + fg + fk + gh + gk + hk + a+b+c+d+e+f+g+h}}{(q)_{a+i} (q)_{a+k} (q)_{e+k} (q)_{f+k} (q)_{g+k} (q)_{h+k} (q)_{b+i+j} (q)_{c+i+j+l}} \\
& (\text{apply (\ref{key}) to the $i$-sum with $n=5$}) \\
& =  \frac{1}{(q)_{\infty}^4} \sum_{a,e,f,g,h,i,k,j,l \geq 0} (-1)^{i+l} \frac{q^{\frac{3i(i+1)}{2} + j^2 + j + \frac{l(l+1)}{2} + 2ij + il + jl + k(3k+2) + ae  + ak + ef + ek + fg + fk + gh + hk}}{(q)_a (q)_e (q)_f (q)_g (q)_h (q)_i (q)_j (q)_k (q)_l (q)_{a+k} (q)_{e+k} (q)_{f+k}} \\
& \times \frac{q^{a+e+f+g+h}}{(q)_{g+k} (q)_{h+k}} \\
& (\text{evaluate the $d$-sum, $c$-sum and $b$-sum with (\ref{e1}) and simplify}) \\
&  = \frac{1}{(q)_{\infty}^4} h_5 \sum_{a,e,f,g,h,k \geq 0} \frac{q^{k(3k+2) + ak + ek + fk + gk + hk + ae + ef + fg + gh + a + e + f + g + h}}{(q)_a (q)_e (q)_f (q)_g (q)_h (q)_k (q)_{a+k} (q)_{e+k} (q)_{f+k} (q)_{g+k} (q)_{h+k}} \\
& (\text{evaluate the $ijl$-sum using (\ref{triple})}).
\end{aligned}
\end{equation*}

\noindent Now, (\ref{-108}) follows from (\ref{72}) after applying $(a,e,f,g,h,k) \to (c,d,e,f,g,a)$. 

For $\Phi_{10_{10}}(q)$, it suffices to prove

\begin{equation} \label{1010}
\begin{aligned}
& S_{10_{10}}(q) := \sum_{a,c,d,e,f,g,h,i,j,k \geq 0} (-1)^{i+j} \frac{q^{\frac{i(3i+1)}{2} + \frac{j(3j+1)}{2} + k(3k+2) + ah + ai + cd + ck + de + dk + ef + ek + fg + fk + gh}}{(q)_a (q)_c (q)_d (q)_e (q)_f (q)_g (q)_h (q)_i (q)_j (q)_k (q)_{a+i} (q)_{c+k} (q)_{d+k}} \\
& \times \frac{q^{gj + gk + hi + hj + a+c+d+e+f+g+h}}{(q)_{e+k} (q)_{f+k} (q)_{g+k} (q)_{g+j} (q)_{h+j} (q)_{h+i}} \\
& = \frac{1}{(q)_{\infty}^8} h_6.
\end{aligned}
\end{equation}

\noindent Apply (\ref{key}) with $n=3$ to the $i$-sum, (\ref{e1}) to the $a$-sum and simplify, (\ref{e2}) to the $i$ and simplify to obtain

\begin{equation*}
\begin{aligned}
& S_{10_{10}}(q) = \frac{1}{(q)_{\infty}} \sum_{c,d,e,f,g,h,j,k \geq 0} (-1)^j \frac{q^{\frac{j(3j+1)}{2} + k(3k+2) + cd + ck + de + dk + ef + ek + fg + fk + gh + gj + gk + hj + c}}{(q)_c (q)_d (q)_e (q)_f (q)_g (q)_h (q)_j (q)_k (q)_{c+k} (q)_{d+k} (q)_{e+k} (q)_{f+k}} \\
& \times \frac{q^{d + e+ f + g + h}}{(q)_{g+k} (q)_{g+j} (q)_{h+j}}.
\end{aligned}
\end{equation*}

\noindent Now, (\ref{1010}) follows from (\ref{82}) after letting $(c,d,e,f,g,h,j,k) \to (h,g,f,e,d,c,b,a)$.

For $\Phi_{10_{15}}(q)$, it suffices to prove

\begin{equation} \label{1015}
\begin{aligned}
& S_{10_{15}}(q)  := \sum_{a,b,c,d,e,g,h,i,j,k \geq 0} (-1)^{i+j} \frac{q^{\frac{i(5i+3)}{2} + \frac{j(5j+3)}{2} + k^2 + ab + ah + ai + bc + bi + bj + cd + cj + de + dj + ej + gh + gi}}{(q)_a (q)_b (q)_c (q)_d (q)_e (q)_g (q)_h (q)_i (q)_j (q)_k (q)_{a+i} (q)_{b+i} (q)_{b+j}} \\
& \times \frac{q^{gk + hi + a+b+c+d+e+g+h}}{(q)_{c+j} (q)_{d+j} (q)_{e+j} (q)_{g+i} (q)_{g+k} (q)_{h+i}} \\
& = \frac{1}{(q)_{\infty}^{10}} h_5^2.
\end{aligned}
\end{equation}

\noindent Apply (\ref{andy}) to the $k$-sum, (\ref{key}) with $n=5$ to the $j$-sum, (\ref{e1}) to the $e$-sum and simplify, to the $d$-sum and simplify and to the $c$-sum and simplify and (\ref{triple}) to obtain

\begin{equation*}
S_{10_{15}}(q) = \frac{1}{(q)_{\infty}^5} h_5 \sum_{a,b,g,h,i \geq 0} (-1)^i \frac{q^{\frac{i(5i+3)}{2} + ab + ah + ai + bi + gh + gi + hi + a + b + g + h}}{(q)_a (q)_b (q)_g (q)_h (q)_i (q)_{a+i} (q)_{b+i} (q)_{g+i} (q)_{h+i}}.
\end{equation*}

\noindent Now, (\ref{1015}) follows from (\ref{51}) after letting $(a,b,g,h,i) \to (c,b,e,d,a)$. 

For $\Phi_{10_{19}}(q)$, it suffices to prove

\begin{equation} \label{1019}
\begin{aligned}
& S_{10_{19}}(q) := \sum_{a,c,d,e,f,g,h,i,j,k \geq 0} (-1)^{j+k} \frac{q^{i(2i+1) + \frac{j(3j+1)}{2} + \frac{k(5k+3)}{2} + ah + ai + cd + ck + de + dek + ef + ek + fg + fk}}{(q)_a (q)_c (q)_d (q)_e (q)_f (q)_g (q)_h (q)_i (q)_j (q)_k (q)_{a+i} (q)_{c+k} (q)_{d+k}} \\
& \times \frac{q^{fj + gh + gi + gj + hi + a + c + d + e + f + g + h}}{(q)_{e+k} (q)_{f+k} (q)_{f+j} (q)_{g+j} (q)_{g+i} (q)_{h+i}} \\
& = \frac{1}{(q)_{\infty}^9} h_4 h_5.
\end{aligned}
\end{equation}

\noindent Apply (\ref{key}) with $n=5$ to the $k$-sum, (\ref{e1}) to the $c$-sum and simplify, to the $d$-sum and simplify and to the $e$-sum and simplify and (\ref{triple}) to obtain

\begin{equation*}
S_{10_{19}}(q) = \frac{1}{(q)_{\infty}^4} \sum_{a,f,g,h,i,j \geq 0} (-1)^j \frac{q^{i(2i+1) + \frac{j(3j+1)}{2} + ah + ai + fg + fj + gh + gi + gj + hi + a + f + g + h}}{(q)_a (q)_f (q)_g (q)_h (q)_i (q)_j (q)_{a+i} (q)_{f+j} (q)_{g+j} (q)_{g+i} (q)_{h+i}}.
\end{equation*}

\noindent Now, (\ref{1019}) follows from (\ref{62}) after letting $(a,f,g,h,i,j) \to (a,d,c,b,f,e)$. 

For $\Phi_{10_{26}}(q)$, it suffices to prove

\begin{equation} \label{1026}
\begin{aligned}
& S_{10_{26}}(q)  := \sum_{a,b,c,e,f,g,h,i,j,k \geq 0} (-1)^i \frac{q^{h(2h+1) + \frac{i(3i+1)}{2} + j^2 + k(2k+1) + ab + ag + ah + ai + bc + bh + ch + ef + ek + fg}}{(q)_a (q)_b (q)_c (q)_e (q)_f (q)_g (q)_h (q)_i (q)_j (q)_k (q)_{a+h} (q)_{a+i} (q)_{b+h}} \\
& \times \frac{q^{fk + gi + gj + gk + a + b + c + e + f + g}}{(q)_{c+h} (q)_{e+k} (q)_{f+k} (q)_{g+i} (q)_{g+j} (q)_{g+k}} \\
& = \frac{1}{(q)_{\infty}^9} h_4^2.
\end{aligned}
\end{equation}

\noindent Apply (\ref{andy}) to the $j$-sum, (\ref{key}) with $n=4$ to the $k$-sum, (\ref{e1}) to the $e$-sum and simplify and to the $f$-sum and simplify and (\ref{double}) to obtain

\begin{equation*}
S_{10_{26}}(q) = \frac{1}{(q)_{\infty}^4} h_4 \sum_{a,b,c,g,h,i \geq 0} (-1)^i \frac{q^{h(2h+1) + \frac{i(3i+1)}{2} + ab + ag + ah + ai + bc + bh + ch + gi + a + b + c + g}}{(q)_a (q)_b (q)_c (q)_g (q)_h (q)_i (q)_{a+h} (q)_{a+i} (q)_{b+h} (q)_{c+h} (q)_{g+i}}.
\end{equation*}

\noindent Now, (\ref{1026}) follows from (\ref{62}) after letting $(a,b,c,g,h,i) \to (c,b,a,d,f,e)$. 

For $\Phi_{10_{28}}(q)$, it suffices to prove

\begin{equation} \label{1028}
\begin{aligned}
& S_{10_{28}}(q) := \sum_{a,b,d,e,f,g,h,i,j,k \geq 0} (-1)^{i+j} \frac{q^{\frac{i(3i+1)}{2} + \frac{j(5j+3)}{2} + k(2k+1) + ab + ah + ai + aj + bi + de + dk + ef + ek + fg + fj}}{(q)_a (q)_b (q)_d (q)_e (q)_f (q)_g (q)_h (q)_i (q)_j (q)_k (q)_{a+i} (q)_{a+j} (q)_{b+i}} \\
& \times \frac{q^{fk + gh + gj + hj + a + b + d + e + f + g + h}}{(q)_{d+k} (q)_{e+k} (q)_{f+j} (q)_{f+k} (q)_{g+j} (q)_{h+j}} \\
& = \frac{1}{(q)_{\infty}^9} h_4 h_5.
\end{aligned}
\end{equation}

\noindent Apply (\ref{key}) with $n=3$ to the $i$-sum, (\ref{e1}) to the $b$-sum and simplify and (\ref{e2}) to the $i$-sum to obtain 

\begin{equation*}
\begin{aligned}
& S_{10_{28}}(q) = \frac{1}{(q)_{\infty}} \sum_{a,d,e,f,g,h,j,k \geq 0} (-1)^j \frac{q^{\frac{j(5j+3)}{2} + k(2k+1) + ah + aj + de + dk + ef + ek + fg + fj + fk + gh + gj + hj}}{(q)_a (q)_d (q)_e (q)_f (q)_g (q)_h (q)_j (q)_k (q)_{a+j} (q)_{d+k} (q)_{e+k} (q)_{f+j}} \\
& \times \frac{q^{a+d+e+f+g+h}}{(q)_{f+k} (q)_{g+j} (q)_{h+j}}.
\end{aligned}
\end{equation*}

\noindent Now, (\ref{1028}) follows from (\ref{-84}) after letting $(a,d,e,f,g,h,j,k) \to (f,a,b,c,d,e,g,h)$. 

For $\Phi_{10_{44}}(q)$, it suffices to prove

\begin{equation} \label{1044}
\begin{aligned}
&S_{10_{44}}(q) := \sum_{a,b,c,e,f,g,h,i,j,k \geq 0} (-1)^{h+j+k} \frac{q^{\frac{h(3h+1)}{2} + i(2i+1) + \frac{j(3j+1)}{2} + \frac{k(3k+1)}{2} + ab + ag + ai + aj + bc + bj + bk + ck}}{(q)_a (q)_b (q)_c (q)_e (q)_f (q)_g (q)_h (q)_i (q)_j (q)_k (q)_{a+i} (q)_{a+j} (q)_{b+j}} \\
& \times \frac{q^{ef + eh + fg + fh + fi + gi + a + b + c + e + f + g}}{(q)_{b+k} (q)_{c+k} (q)_{e+h} (q)_{f+h} (q)_{f+i} (q)_{g+i}} \\
& = \frac{1}{(q)_{\infty}^7} h_4. 
\end{aligned}
\end{equation}

\noindent Apply (\ref{key}) with $n=3$ to the $h$-sum, (\ref{e1}) to the $e$-sum and simplify, (\ref{e2}) to the $h$-sum, (\ref{key}) with $n=3$ to the $k$-sum, (\ref{e1}) to the $c$-sum and simplify and (\ref{e2}) to the $k$-sum to obtain

\begin{equation*}
S_{10_{44}}(q) = \frac{1}{(q)_{\infty}^2} \sum_{a,b,f,g,i,j \geq 0} (-1)^j \frac{q^{i(2i+1) + \frac{j(3j+1)}{2} + ab + ag + ai + aj + bj + fg + fi + gi + a + b + f + g}}{(q)_a (q)_b (q)_f (q)_g (q)_i (q)_j (q)_{a+i} (q)_{a+j} (q)_{b+j} (q)_{f+i} (q)_{g+i}}.
\end{equation*}

\noindent Now, (\ref{1044}) follows from (\ref{62}) after letting $(a,b,f,g,i,j) \to (c,d,a,b,f,e)$. 

\end{proof}

\section*{Acknowledgements}
The authors would like to thank Stavros Garoufalidis, Mustafa Hajij and Jeremy Lovejoy for their suggestions and helpful comments. The second author would like to thank Don Zagier for his question which motivated this paper. This question was posed on August 19, 2014 at the workshop ``Low-dimensional topology and number theory" in Oberwolfach. Finally, the second author thanks the organizers (in particular, Frits Beukers) of the conference ``Automatic sequences, Number Theory, Aperiodic Order", October 28--30, 2015 at TU Delft for the opportunity to discuss these results.

\end{document}